\documentclass{amsart}
\usepackage{amsfonts}
\usepackage{bbm}
\usepackage{mathrsfs}
\usepackage{color}
\usepackage{amsmath}
\usepackage{bigdelim}
\usepackage{multirow}
\usepackage{amssymb}
\usepackage{indentfirst}
\usepackage{CJK}
\usepackage{fancyhdr}
\usepackage{amscd,amssymb,amsmath,graphicx,verbatim,xy}
\usepackage[TS1,OT1,T1]{fontenc}

\newtheorem{theorem}{Theorem}[section]
\newtheorem{lemma}[theorem]{Lemma}

\newtheorem{proposition}[theorem]{Proposition}

\theoremstyle{definition}
\newtheorem{definition}[theorem]{Definition}

\theoremstyle{remark}
\newtheorem{remark}[theorem]{Remark}
\numberwithin{equation}{section}

\begin{document}
	
	\title[Diagonalization of skew-Hermitian ]{Diagonalization of skew-Hermitian dual
		quaternion matrices and $\phi$-Hermitian dual quaternion matrices}
	\author{Cailing Yao}
	\address{Cailing Yao, School of Medical Informatics, Changchun University of Chinese Medicine, 130117, Changchun, P. R. China}
	\email{yaocl@ccucm.edu.cn}
	\author{Bingzhe Hou}
	\address{Bingzhe Hou, School of Mathematics, Jilin University, 130012, Changchun, P. R. China}
	\email{houbz@jlu.edu.cn}

	
	
	\begin{abstract}
In this article, we present the diagonalization of skew-Hermitian dual quaternion matrices in the sense of unitary equivalence. For a unit quaternion $\phi$, we also obtain the diagonalization result of $\phi$-Hermitian dual quaternion matrices in the sense of unitary $\phi$-congruence, which generalizes the Autonne--Takagi factorization of symmetric complex matrices and $\phi$-Hermitian quaternion matrices.
	\end{abstract}
	\maketitle

\section{Introduction}

The skew field $\mathbb{Q}$ of quaternions was discovered by W. Hamilton \cite{H1843} in 1843. Quaternions have extensive applications in many fields, such as physics \cite{WW}, computer graphics \cite{KS}, aerospace \cite{WGM}, virtual reality \cite{KC}, and so on. Dual quaternions were introduced by Clifford \cite{C} in $1873$. By now, dual quaternions have been widely applied in many fields, such as robotics \cite{AC}, 3D motion modelling and control \cite{LL}. The theories related to complex matrices and quaternion matrices have already been well developed, including spectral theory, decomposition theory, numerical analysis, generalized inverses and equations. However, the corresponding spectral and factorization theories for dual quaternion matrices remain incomplete. It is worth mentioning that L. Qi and Z. Luo \cite{QL} explored the spectral theory of dual quaternion matrices. In particular, the diagonalization problem is a natural and important topic.

R. Horn and F. Zhang \cite{HZ} extended the Autonne--Takagi factorization of symmetric complex matrices to $\eta$-Hermitian quaternion matrices, where $\eta$ is a unit pure imaginary quaternion. C. Yao et al. \cite{YHX} further extended this result from unit pure imaginary quaternions to arbitrary unit quaternions and introduced the notions of $\phi$-conjugation and unitary $\phi$-congruence. The notion of $\phi$-conjugation was motivated by the $\boldsymbol{j}$-conjugate introduced in \cite{H}, the $\boldsymbol{i}$-conjugate introduced in \cite{YHF} and conjugate-normality studied in \cite{FI08}. The concept of unitary $\phi$-congruence was inspired by the classical notion of unitary congruence in \cite{HLM1974}. L. Qi and Z. Luo \cite{QL} established a unitary diagonalization result for Hermitian dual quaternion matrices. 

Throughout this paper, unitary equivalence refers to a transformation of the form $A\mapsto UAU^{\ast}$, where $U$ is unitary. Motivated by the aforementioned results, we present the diagonalization conclusion of skew-Hermitian dual quaternion matrices under unitary equivalence. We also prove that for every unit quaternion $\phi$, each $\phi$-Hermitian dual quaternion matrix can be diagonalized through unitary $\phi$-congruence. Relatedly, R. Xu et al. \cite{XW} obtained an Autonne–Takagi factorization for dual symmetric complex matrices and established the diagonalization of \(\eta\)-Hermitian dual quaternion matrices for any unit pure imaginary quaternion \(\eta\). The transformation used therein is \(A\mapsto VAV^{\eta H}\), where \(H\) denotes the conjugate transpose and \(V^{\eta H}=\eta^{H}V^{H}\eta\). The matrix \(V\) satisfies \(VV^{\eta H}=I_n\), rather than the usual unitary condition \(VV^{H}=I_n\). Consequently, this transformation is generally not a unitary equivalence in the classical sense.

Notice that the diagonal entries of the diagonal matrix obtained from skew-Hermitian quaternion matrices under unitary equivalence are all purely imaginary numbers, and purely imaginary quaternions do not in general commute with quaternion matrices, even when the standard diagonal block is scalar. Then, the overall procedure for diagonalizing skew-Hermitian dual quaternion matrices becomes more complicated.

In Section $2$, we mainly introduce some relevant concepts and their basic properties that will be needed in the later discussion including dual numbers, quaternions and quaternion matrices, dual quaternions. In Section $3$, we provide a detailed proof of how to diagonalize a skew-Hermitian dual quaternion matrix in the sense of unitary equivalence. Throughout the diagonalization process, the proof identifies why the diagonalization of a skew-Hermitian dual quaternion matrix in the sense of unitary equivalence is more complicated than that of a Hermitian dual quaternion matrix. In Section $4$, the diagonalization of a $\phi$-Hermitian dual quaternion matrix in the sense of unitary $\phi$-congruence is provided, where $\phi$ can range over all unit quaternions, not only unit pure imaginary quaternions. This kind of diagonalization can be viewed as a generalization of the Autonne--Takagi factorization of symmetric complex matrix and $\phi$-Hermitian quaternion matrix for $\phi\in\mathbb{S}^{3}$.

\section{Preliminaries}
In this section, we will introduce some basic information about dual numbers, quaternions, dual quaternions and so on.

\noindent\textbf{Dual numbers.} \
 \ Denote by $\mathbb{R}$, $\mathbb{C}$ and $\mathbb{Q}$ the sets of all real numbers, complex numbers and quaternions, respectively. Assume $q=q_{st}+q_{\mathcal{I}}\epsilon$, where $q_{st},q_{\mathcal{I}}\in\mathbb{R}$, $\epsilon$ is the infinitesimal unit, satisfying $\epsilon^{2}=0$. We call $q$ a dual number, $q_{st}$ the standard part of $q$ and $q_{\mathcal{I}}$ the dual part or infinitesimal part of $q$. Denote by $\mathbb{D}$ the set of dual numbers. In addition, if $q_{st},q_{\mathcal{I}}\in\mathbb{C}$ or $q_{st},q_{\mathcal{I}}\in\mathbb{Q}$, respectively, then we say $q$ is a dual complex number or dual quaternion. We denote by $\mathbb{DC}$ the set of dual complex numbers and $\mathbb{DQ}$ the set of dual quaternions, respectively. The infinitesimal unit $\epsilon$ is commutative in multiplication with numbers in $\mathbb{R}$, $\mathbb{C}$ and $\mathbb{Q}$. If $q_{st}\neq 0$, we say $q$ is appreciable, or else, we say $q$ is infinitesimal.

For any two dual numbers $p,q\in \mathbb{D}$, where
$$
p=p_{st}+p_{\mathcal{I}}\epsilon, \ q=q_{st}+q_{\mathcal{I}}\epsilon.
$$
 If either of the following two cases occurs,
\begin{enumerate}
\item [(i).] $p_{st}<q_{st}$,

\item [(ii).] $p_{st}=q_{st}$ and $p_{\mathcal{I}}\leq q_{\mathcal{I}}$,
\end{enumerate}
then we say $p\leq q$. Furthermore, we say $p$ is positive, nonnegative, nonpositive or negative, if $p>0$, $p\geq 0$, $p\leq 0$ or $p<0$, respectively. More about the total order for dual numbers can be seen in \cite{QLY}.

\noindent\textbf{Quaternions and quaternion matrices.}\ \ We define $a\in\mathbb{Q}$ to be $a=a_0+a_1\boldsymbol{i}+a_2\boldsymbol{j}+a_3\boldsymbol{k},\ a_0, a_1, a_2, a_3\in\mathbb{R}$ and $\boldsymbol{i}^2=\boldsymbol{j}^2=\boldsymbol{k}^2=\boldsymbol{ijk}=-1$. Denote the real and imaginary parts of $a$  by $\mathrm{Re}(a)=a_0$ and  $\mathrm{Im}(a)=a_1\boldsymbol{i}+a_2\boldsymbol{j}+a_3\boldsymbol{k},$ respectively, the modulus of $a$ by
$|a|=\sqrt{a_0^2+a_1^2+a_2^2+a_3^2}$, and the conjugation of $a$ by $\overline{a}=a_0-a_1\boldsymbol{i}-a_2\boldsymbol{j}-a_3\boldsymbol{k}$. We say $a\in\mathbb{Q}$ is invertible if there exists $b\in\mathbb{Q}$ such that $ab=ba=1$ and call $b$ the inverse of $a$.
Assume $\phi\in\mathbb{Q}$, then $\phi$ can be expressed in the form $\phi_{0}+\phi_{1}\boldsymbol{i}+\phi_{2}\boldsymbol{j}+\phi_{3}\boldsymbol{k}$.
Let
$$
\mathbb{S}^{3}=\{ \phi\in\mathbb{Q}; |\phi|=1\}, \ \mathbb{S}=\{ \phi\in\mathbb{Q}; |\phi|=1, \phi_{0}=0\}.
$$
Then $\mathbb{S}^{3}\setminus \mathbb{S}=\{ \phi\in \mathbb{S}^{3}; \phi_{0}\neq 0\}$.

Denote by $\mathbb{C}^{n\times n}$ and $\mathbb{Q}^{n\times n}$ the set of all $n\times n$ complex matrices and $n\times n$ quaternion matrices, respectively. If $Q=(q_{ij}) \in \mathbb{Q}^{n\times n}$, we use $\overline{Q}=(\overline{q_{ij}}),Q^{T}=(q_{ji}),Q^{*}=\overline{Q^{T}}=(\overline{Q})^{T}$ to denote the conjugation, transpose and conjugate transpose of $Q$ respectively. If $Q=\overline{Q}$, then $Q$ is a real matrix. If $Q=Q^{T}$, we say $Q$ is symmetric and if $Q=Q^{*}(Q=-Q^{*})$, then we say $Q$ is Hermitian (skew-Hermitian).

For two matrices $A,B\in \mathbb{Q}^{n\times n}$, if there exists an invertible matrix $P\in \mathbb{Q}^{n\times n}$ such that $PAP^{-1}=B$, then we say $A$ is similar to $B$ and denote by $A\sim B$. In particular, when the invertible matrix turns to a unitary matrix $U$, we say $A$ is unitarily equivalent to $B$ and denote by $A\sim_{u} B$.

\noindent\textbf{Dual quaternions.}\ \ For a dual quaternion $q\in\mathbb{DQ}$, it has the form $q=q_{st}+q_{\mathcal{I}}\epsilon$, where $q_{st},q_{\mathcal{I}}\in\mathbb{Q}$ are the standard part and the infinitesimal part of $q$, respectively. The conjugation of $q$ is defined to be $q^{\ast}=q_{st}^{\ast}+q_{\mathcal{I}}^{\ast}\epsilon$. Hence, if $q=q^{\ast}$, then $q\in\mathbb{D}$.

Now we let $p,q\in\mathbb{DQ}$, where $p=p_{st}+p_{\mathcal{I}}\epsilon$ and $q=q_{st}+q_{\mathcal{I}}\epsilon$. Define their addition and multiplication as
$$
p+q=(p_{st}+q_{st})+(p_{\mathcal{I}}+q_{\mathcal{I}})\epsilon,~ pq=p_{st}q_{st}+(p_{st}q_{\mathcal{I}}+p_{\mathcal{I}}q_{st})\epsilon.
$$
It is not hard to verify that a dual number is commutative with any dual quaternion. For $p\in\mathbb{DQ}$, if there exists $q\in\mathbb{DQ}$ such that $pq=qp=1$, then we say $p$ is invertible and $q$ is the inverse of $p$. It follows that $p$ is invertible if and only if $p$ is appreciable and
$$
p^{-1}=p_{st}^{-1}-p_{st}^{-1}p_{\mathcal{I}}p_{st}^{-1}\epsilon.
$$
For $p=p_{st}+p_{\mathcal{I}}\epsilon$, the modulus of $p$ is defined as
\begin{enumerate}
\item if $p_{st}=0$, $|p|=|p_{\mathcal{I}}|\epsilon$, where $|p_{\mathcal{I}}|$ denotes the modulus of the quaternion $p_{\mathcal{I}}$. That is, if we set $p_{\mathcal{I}}=y_{0}+y_{1}\boldsymbol{i}+y_{2}\boldsymbol{j}+y_{3}\boldsymbol{k}$, then
    $$
    |p_{\mathcal{I}}|=\sqrt{y_{0}^{2}+y_{1}^{2}+y_{2}^{2}+y_{3}^{2}};
    $$
\item if $p_{st}\neq 0$, $|p|=|p_{st}|+\frac{p_{st}\cdot p_{\mathcal{I}}}{|p_{st}|}\epsilon$, where $p_{st}\cdot p_{\mathcal{I}}$ is the dot product of $p_{st}$ and $p_{\mathcal{I}}$. Specifically, if we let $p_{st}=x_{0}+x_{1}\boldsymbol{i}+x_{2}\boldsymbol{j}+x_{3}\boldsymbol{k}$, $p_{\mathcal{I}}=y_{0}+y_{1}\boldsymbol{i}+y_{2}\boldsymbol{j}+y_{3}\boldsymbol{k}$, then
$$
p_{st}\cdot p_{\mathcal{I}}=x_{0}y_{0}+x_{1}y_{1}+x_{2}y_{2}+x_{3}y_{3}.
$$
Direct calculation gives $p_{st}\cdot p_{\mathcal{I}}=\frac{p_{st}p_{\mathcal{I}}^{\ast}+p_{\mathcal{I}}p_{st}^{\ast}}{2}$ and by the definition of modulus for $p\in \mathbb{DQ}$, we can see that $|p|$ is a nonnegative dual number. If $|p|=1$, then $p$ is called a unit dual quaternion. A unit dual quaternion $p$ must be invertible and $p^{-1}=p^{\ast}$ always holds. Further, $|p|=1$ if and only if $|p_{st}|=1$ and $p_{st}\cdot p_{\mathcal{I}}=0$.
\end{enumerate}

\noindent\textbf{Eigenvalues of dual quaternion matrices.}\ \ We use notations $\mathbb{DC}^{n\times n}$ and $\mathbb{DQ}^{n\times n}$ to denote the set of $n\times n$ dual complex matrices and $n\times n$ dual quaternion matrices. For a matrix $A$ in $\mathbb{DQ}^{n\times n}$, notations $\overline{A}$, $A^{T}$ and $A^{\ast}$ still denote the conjugation, transpose and conjugate transpose of $A$, respectively. More precisely, if we set $A=A_{st}+A_{\mathcal{I}}\epsilon$, then
$$
\overline{A}=\overline{A_{st}}+\overline{A_{\mathcal{I}}}\epsilon,\ A^{T}=A_{st}^{T}+A_{\mathcal{I}}^{T}\epsilon,\ A^{\ast}=A_{st}^{\ast}+A_{\mathcal{I}}^{\ast}\epsilon.
$$

In \cite{QL}, the authors for the first time provided the definitions of left and right eigenvalues for dual quaternion matrices in $\mathbb{DQ}^{n\times n}$ and investigated their related properties. Suppose that $A\in \mathbb{DQ}^{n\times n}$. If there exist $\lambda\in\mathbb{DQ}$ and $\boldsymbol{x}\in\mathbb{DQ}^{n}$, where $\boldsymbol{x}$ is appreciable, such that
$$
A\boldsymbol{x}=\boldsymbol{x}\lambda,
$$
then we say that $\lambda$ is a right eigenvalue of $A$, with $\boldsymbol{x}$ as an associated right eigenvector. They proved in \cite{QL} that a right eigenvalue $\lambda$ of a Hermitian matrix $A\in\mathbb{DQ}^{n\times n}$ must be a dual number and a dual quaternion Hermitian matrix has at most $n$ dual number eigenvalues and no other right eigenvalues.

\section{Unitary diagonalization of skew-Hermitian dual quaternion matrices}
In this section, we study the diagonalization of skew-Hermitian dual quaternion matrices under unitary equivalence. First of all, we introduce and prove some conclusions, which will be helpful for proving our main theorem in this section.

\begin{lemma}\label{264272}
Let $N\in\mathbb{C}^{n\times n}$. Then $N$ is a symmetric complex matrix if and only if $\boldsymbol{j}N$ is a skew-Hermitian quaternion matrix.
\end{lemma}
\begin{proof}
Note that for $N\in\mathbb{C}^{n\times n}$, $\boldsymbol{j}N=\overline{N}\boldsymbol{j}$. The fact that $\boldsymbol{j}N$ is skew-Hermitian is equivalent to $N^{\ast}\boldsymbol{j}=\boldsymbol{j}N(=\overline{N}\boldsymbol{j})$. Then, $\boldsymbol{j}N$ is skew-Hermitian if and only if
$N^{T}=N$, i.e., $N$ is symmetric complex.
\end{proof}
\begin{lemma}[\cite{HZ}]\label{264271}
Assume $A\in\mathbb{Q}^{n\times n}$ is skew-Hermitian. Then there is a unitary quaternion matrix $U$ and a diagonal matrix $\Lambda$ such that
\begin{equation}\label{88}
UAU^{\ast}=\Lambda=\mathrm{diag}(r_{1}I_{k_{1}}\boldsymbol{i},r_{2}I_{k_{2}}\boldsymbol{i},\ldots,r_{r}I_{k_{r}}\boldsymbol{i}),
\end{equation}
where $0\leq r_{1}<\cdots<r_{r}$, $I_{k_{j}}$ is a $k_{j}\times k_{j}$ identity matrix and $\sum_{j=1}^{r}k_{j}=n$.
\end{lemma}

\begin{lemma}[\cite{FZZ}]\label{264281}
If $A\in\mathbb{C}^{n\times n}$ is symmetric, then there exists a unitary matrix $U\in\mathbb{C}^{n\times n}$ and a nonnegative diagonal matrix $\Upsilon$ such that
$$
U^{T}AU=\Upsilon.
$$
\end{lemma}
\begin{lemma}[\cite{FZZ}]\label{12}
For any complex matrix $C$, the equality $C\boldsymbol{j}=\boldsymbol{j}\overline{C}$ holds.
\end{lemma}

\begin{lemma}[\cite{Z}]\label{1}
Assume $p,q\in\mathbb{Q}$. Then $p\sim q$ if and only if $|p|=|q|$ and $p_{0}=q_{0}$.
\end{lemma}
\begin{lemma}\label{2}
Let $p=a\boldsymbol{i}+(b\boldsymbol{i}+c\boldsymbol{j})\epsilon\in\mathbb{DQ}$, where $a,b,c\in\mathbb{R}$. Then $p$ is unitarily equivalent to a dual complex number.
\end{lemma}
\begin{proof}
If $a=0$, then $p=(b\boldsymbol{i}+c\boldsymbol{j})\epsilon$. By Lemma \ref{1}, the infinitesimal part $b\boldsymbol{i}+c\boldsymbol{j}$ is similar to the complex number $\sqrt{b^{2}+c^{2}}\boldsymbol{i}\in\mathbb{C}\subseteq\mathbb{DC}$. That is, there is a nonzero quaternion $s$ such that
$$
sps^{-1}=\sqrt{b^{2}+c^{2}}\boldsymbol{i}\epsilon.
$$
In particular, if we take $s$ to have unit norm, then the similarity becomes a unitary equivalence.

If $a\neq 0$, let $u=1+\frac{c\boldsymbol{k}}{2a}\epsilon$ and $v=\frac{u}{|u|}$. Then $v^{\ast}v=vv^{\ast}=1$ and
$$
v^{\ast}pv=v^{\ast}(a\boldsymbol{i})v+v^{\ast}(b\boldsymbol{i}+c\boldsymbol{j})v\epsilon=\frac{a\boldsymbol{i}+b\boldsymbol{i}\epsilon}{|u|^{2}}\in\mathbb{DC}.
$$
Therefore, whether $a$ is zero or not, $p$ is unitarily equivalent to a dual complex number.
 \end{proof}
 The result in the next Lemma may be well-known, its one-dimensional case was  mentioned in \cite{QWL} and it is easy to obtain the proof by straightforward computation.
\begin{lemma}\label{99}
Let $P=P_{st}+P_{\mathcal{I}}\epsilon\in\mathbb{DQ}^{n\times n}$. Then $P$ is unitary in $\mathbb{DQ}^{n\times n}$ if and only if $P_{st}$ is unitary $\mathbb{Q}^{n\times n}$ and $P_{st}P_{\mathcal{I}}^{\ast}$ is skew-Hermitian. 
 \end{lemma}
\begin{theorem}\label{4}
Suppose that $A=A_{st}+A_{\mathcal{I}}\epsilon\in\mathbb{DQ}^{n\times n}$ is a skew-Hermitian matrix. Then there exists a unitary matrix $W\in\mathbb{DQ}^{n\times n}$ and a diagonal matrix $\Sigma\in\mathbb{DQ}^{n\times n}$ such that
$$
\Sigma=WAW^{\ast},
$$
 where
\begin{equation}\label{9}
\Sigma=\mathrm{diag}(r_{1}\boldsymbol{i}+R_{11}\boldsymbol{i}\epsilon,\ldots,r_{1}\boldsymbol{i}+R_{1k_{1}}\boldsymbol{i}\epsilon,r_{2}\boldsymbol{i}+R_{21}\boldsymbol{i}\epsilon,\ldots,r_{r}\boldsymbol{i}+R_{rk_{r}}\boldsymbol{i}\epsilon).
\end{equation}
Each diagonal entry of $\Sigma$ is a right eigenvalue of $A$ and
$$
0\leq r_{1}<\cdots<r_{r},\ \  R_{i1}\leq\cdots\leq R_{ik_{i}},\ \ R_{ij}\in\mathbb{R},\ \ \sum_{i=1}^{r}k_{i}=n.
$$  
\end{theorem}
\begin{proof}
Denote $A=A_{st}+A_{\mathcal{I}}\epsilon\in\mathbb{DQ}^{n\times n}$, where $A_{st},A_{\mathcal{I}}\in\mathbb{Q}^{n\times n}$. Then $A$ is  a skew-Hermitian dual quaternion matrix if and only if both $A_{st}$ and $A_{\mathcal{I}}$ are skew-Hermitian quaternion matrices. It follows from Lemma \ref{264271} that there is a unitary quaternion matrix $U\in\mathbb{Q}^{n\times n}$ and a diagonal matrix $\Lambda$ with the form in (\ref{88}) such that
$$
\Lambda=UA_{st}U^{\ast}.
$$
 Then
$$
\begin{aligned}
A_{1}&=UAU^{\ast}=\Lambda+UA_{\mathcal{I}}U^{\ast}\epsilon\\
&=\begin{pmatrix}
r_{1}I_{k_{1}}\boldsymbol{i}+C_{11}\epsilon & C_{12}\epsilon &  \cdots     &C_{1r}\epsilon  \\
-C_{12}^{\ast}\epsilon  &r_{2}I_{k_{2}}\boldsymbol{i}+C_{22}\epsilon &   \cdots    &  C_{2r}\epsilon\\
 \vdots & \vdots &\ddots & \vdots \\
 -C_{1r}^{\ast}\epsilon & -C_{2r}^{\ast}\epsilon &  \cdots     &r_{r}I_{k_{r}}\boldsymbol{i}+C_{rr}\epsilon
\end{pmatrix},
\end{aligned}
$$
where each $C_{ij}$ is a quaternion matrix of matching dimension and each $C_{ii}$ is a skew-Hermitian quaternion matrix. Let $C_{ij}=M_{ij}+\boldsymbol{j}N_{ij}$, where $M_{ij}$ and $N_{ij}$ are complex matrices with the same dimensions as $C_{ij}$. Define
$$
P=\begin{pmatrix}
P_{11}           &  P_{12} & \cdots & P_{1r}\\
-P_{12}^{\ast}   &  P_{22} & \cdots & P_{2r}\\
\vdots           &  \vdots &\ddots  & \vdots\\
-P_{1r}^{\ast}   &  -P_{2r}^{\ast} &\cdots & P_{rr}
\end{pmatrix},
$$
where for $1\leq s,t\leq r$, each $P_{st}$ is defined as follows,
\begin{enumerate}
\item [(1).]if $s=t$,  $P_{st}=I_{k_{s}}=I_{k_{t}}$;
\item [(2).] if $s<t$,  $P_{st}=\frac{M_{st}}{r_{s}\boldsymbol{i}-r_{t}\boldsymbol{i}}\epsilon+(r_{s}\boldsymbol{i}+r_{t}\boldsymbol{i})^{-1}\boldsymbol{j}N_{st}\epsilon$;

\item [(3).] if $s>t$,
$P_{st}=-(P_{ts})^{\ast}=
\frac{M_{ts}^{\ast}}{r_{t}\boldsymbol{i}-r_{s}\boldsymbol{i}}\epsilon-N_{ts}^{\ast}\boldsymbol{j}(r_{t}\boldsymbol{i}+r_{s}\boldsymbol{i})^{-1}\epsilon$.
\end{enumerate}
Since the matrix $P$ is the sum of the identity matrix $I_n$ and a skew-Hermitian matrix, it follows from Lemma \ref{99} that  $P$ is unitary and
$$
\begin{aligned}
A_{2}=PA_{1}P^{\ast}&=\mathrm{diag}(r_{1}I_{k_{1}}\boldsymbol{i}+C_{11}\epsilon,\ldots,r_{r}I_{k_{r}}\boldsymbol{i}+C_{rr}\epsilon)\\
&=\mathrm{diag}(r_{1}I_{k_{1}}\boldsymbol{i}+M_{11}\epsilon+\boldsymbol{j}N_{11}\epsilon,\ldots,r_{r}I_{k_{r}}\boldsymbol{i}+M_{rr}\epsilon+\boldsymbol{j}N_{rr}\epsilon).
\end{aligned}
$$
Notice that each $M_{ii}$ is a skew-Hermitian complex matrix and each $N_{ii}$ is a symmetric complex matrix by Lemma \ref{264272}. Thus we can find unitary complex matrices $V_{11},\ldots,V_{rr}$ such that each
$$
\Lambda_{ii}=V_{ii}M_{ii}V_{ii}^{\ast}
$$
has the same block form as the matrix $\Lambda$ in Lemma \ref{264271}. That is,
\begin{equation}\label{264282}
\Lambda_{ii}=\mathrm{diag}(\rho_{i1}I_{l_{i1}}\boldsymbol{i},\ldots,\rho_{id_{i}}I_{l_{id_{i}}}\boldsymbol{i}),
\end{equation}
where $\rho_{i1}<\cdots<\rho_{id_{i}},\ l_{i1}+\cdots+l_{id_{i}}=k_{i}$. For each $i$ such that $r_i>0$, set
\[
(R_{i1},\ldots,R_{ik_i})
=
\bigl(
\underbrace{\rho_{i1},\ldots,\rho_{i1}}_{l_{i1}},
\ldots,
\underbrace{\rho_{id_i},\ldots,\rho_{id_i}}_{l_{id_i}}
\bigr).
\]
Denote $V=\mathrm{diag}(V_{11},\ldots,V_{rr})$. Then $V$ is also a unitary complex matrix in $\mathbb{C}^{n\times n}$ and
$$
A_{3}=VA_{2}V^{\ast}=\mathrm{diag}(r_{1}I_{k_{1}}\boldsymbol{i}+\Lambda_{11}\epsilon+\boldsymbol{j}E^{11}\epsilon,\ldots,r_{r}I_{k_{r}}\boldsymbol{i}+\Lambda_{rr}\epsilon+\boldsymbol{j}E^{rr}\epsilon)
$$
is also skew-Hermitian. Therefore, each $E^{ii}$ is symmetric complex. We partition the matrix $E^{ii}$ according to the same block form as matrix $\Lambda_{ii}$ in (\ref{264282}). That is,

\begin{equation}\label{tt}
E^{ii}=\begin{pmatrix}
 E^{ii}_{11}&            E^{ii}_{12} &  \cdots     & E^{ii}_{1d_{i}}  \\
(E^{ii}_{12})^{T}&         E^{ii}_{22} &  \cdots     & E^{ii}_{2d_{i}}   \\
 \vdots        &         \vdots      &  \ddots     & \vdots        \\
(E^{ii}_{1d_{i}})^{T} &       (E^{ii}_{2d_{i}})^{T} &  \cdots     & E^{ii}_{d_{i}d_{i}}
\end{pmatrix},
\end{equation}
where \(E^{ii}_{st}\in\mathbb C^{l_{is}\times l_{it}}\) for
\(1\leq s,t\leq d_i\).

We now focus on block-diagonalizing $A_{3}$ under unitary equivalence, where the block structure follows that of matrix $\Lambda_{ii}$. One thing to note here is that, during the calculation, we consider the following two cases depending on whether $r_{1}$ is zero or not.
\begin{enumerate}
\item [(1).]$r_{1}\neq 0$. For integers $1\leq h\leq r$, let
$$
Q_{h}=\begin{pmatrix}
I_{l_{h1}} &           (2r_{h}\boldsymbol{i})^{-1}\boldsymbol{j} E^{hh}_{12}\epsilon &  \cdots     & (2r_{h}\boldsymbol{i})^{-1}\boldsymbol{j} E^{hh}_{1d_{h}}\epsilon  \\
-(E^{hh}_{12})^{\ast}\boldsymbol{j}(2r_{h}\boldsymbol{i})^{-1}\epsilon&         I_{l_{h2}} &  \cdots     & (2r_{h}\boldsymbol{i})^{-1}\boldsymbol{j}E^{hh}_{2d_{h}}\epsilon   \\
 \vdots        &         \vdots      &  \ddots     & \vdots        \\
-(E^{hh}_{1d_{h}})^{\ast}\boldsymbol{j}(2r_{h}\boldsymbol{i})^{-1}\epsilon &       -(E^{hh}_{2d_{h}})^{\ast}\boldsymbol{j}(2r_{h}\boldsymbol{i})^{-1}\epsilon &  \cdots &   I_{l_{hd_{h}}}
\end{pmatrix}.
$$
It follows from the form of $Q_h$ and Lemma \ref{99} that each $Q_{h}$ is unitary. Denote
$$
Q=\begin{pmatrix}
 Q_{1}&     &          \\
&         \ddots &         \\
  & &  Q_{r}
\end{pmatrix},
$$
then $Q$ is a unitary matrix in $\mathbb{DQ}^{n\times n}$. And furthermore,
$$
A_{4}=QA_{3}Q^{\ast}=\begin{pmatrix}
 D_{1}&     &          \\
&         \ddots &         \\
  & &  D_{r}
\end{pmatrix},
$$
where for $1\leq p\leq r$, the block matrix
$$
D_{p}=\begin{pmatrix}
r_{p}I_{l_{p1}}\boldsymbol{i}+\rho_{p1}I_{l_{p1}}\boldsymbol{i}\epsilon+\boldsymbol{j}E^{pp}_{11}\epsilon &     &          \\
&         \ddots &         \\
  & &  r_{p}I_{l_{pd_{p}}}\boldsymbol{i}+\rho_{pd_{p}}I_{l_{pd_{p}}}\boldsymbol{i}\epsilon+\boldsymbol{j}E^{pp}_{d_{p}d_{p}}\epsilon
\end{pmatrix}.
$$

Clearly, $A_{4}$ is a block diagonal matrix.

\item [(2).]$r_{1}= 0$. 
 For $2\leq h\leq r$, we repeat the process in the
case $r_1\neq0$. Now, the only block that remains to be handled is
the one containing $r_1$. This block becomes
$\Lambda_{11}\epsilon+\boldsymbol{j}E^{11}\epsilon$.
Notice that $\Lambda_{11}+\boldsymbol{j}E^{11}$ is a skew-Hermitian quaternion
matrix. Then, by Lemma 3.2, there exists a unitary quaternion matrix
$Q_1\in\mathbb{Q}^{k_1\times k_1}$ such that
\[
Q_1(\Lambda_{11}+jE^{11})Q_1^*
   =\operatorname{diag}(y_1\boldsymbol{i},\ldots,y_{k_1}\boldsymbol{i}),
\qquad
0\leq y_1\leq\cdots\leq y_{k_1}.
\]
Together with the matrices $Q_h$, $2\leq h\leq r$, constructed as in
case (1), let
\[
Q=\operatorname{diag}(Q_1,Q_2,\ldots,Q_r).
\]
Then $A_4=QA_3Q^*$ is block diagonal, and its first diagonal block is
\[
D_1=\operatorname{diag}
(y_1\boldsymbol{i}\epsilon,\ldots,y_{k_1}\boldsymbol{i}\epsilon).
\]
In this case, set
\[
R_{1\ell}=y_\ell,\qquad \ell=1,\ldots,k_1.
\]
\end{enumerate}
Therefore, in either case, we have successfully transformed $A_{3}$ into the block-diagonal matrix $A_{4}$. Now for $A_{4}$, since each $E^{ii}_{qq}$ is symmetric complex, we may find a unitary complex matrix $Z^{i}_{q}\in\mathbb{C}^{l_{iq}\times l_{iq}}$ such that each
$$
\Theta^{i}_{q}=(Z^{i}_{q})^{T}E^{ii}_{qq}Z^{i}_{q}
$$
is nonnegative diagonal by Lemma \ref{264281}.  Here $1\leq i\leq r$ when
$r_1\neq0$, whereas $2\leq i\leq r$ when $r_1=0$.
If $r_1\neq0$, let
\[
Z=\operatorname{diag}
(Z_1^1,\ldots,Z_{d_1}^1,\ldots,
 Z_1^r,\ldots,Z_{d_r}^r).
\]
If $r_1=0$, let
\[
Z=\operatorname{diag}
(I_{k_1},Z_1^2,\ldots,Z_{d_2}^2,\ldots,
 Z_1^r,\ldots,Z_{d_r}^r).
\]
If $r=1$ and $r_1=0$, we simply take $Z=I_{k_1}$. Then $Z$ is a unitary matrix in $\mathbb{C}^{n\times n}$. It then follows from Lemma \ref{12} that,
$$
A_{5}=Z^{\ast}A_{4}Z=\begin{pmatrix}
 G_{1}&     &          \\
&         \ddots &         \\
  & &  G_{r}
\end{pmatrix}.
$$
If $r_1=0$, the first block is already diagonal:
\[
G_1=D_1
=\operatorname{diag}
(y_1\boldsymbol{i}\epsilon,\ldots,y_{k_1}\boldsymbol{i}\epsilon).
\]
All the remaining blocks, as well as all the blocks in the case
$r_1\neq0$, have the form
\begin{equation}
G_{t}=\begin{pmatrix}
r_{t}I_{l_{t1}}\boldsymbol{i}+\rho_{t1}I_{l_{t1}}\boldsymbol{i}\epsilon+\boldsymbol{j}\Theta^{t}_{1}\epsilon &  &                 \\
     &  \ddots      &              \\
   &  &  &r_{t}I_{l_{td_{t}}}\boldsymbol{i}+\rho_{td_{t}}I_{l_{td_{t}}}\boldsymbol{i}\epsilon+\boldsymbol{j}\Theta^{t}_{d_{t}}\epsilon
\end{pmatrix}.
 \end{equation}
Here $2\leq t\leq r$ if $r_1=0$, whereas $1\leq t\leq r$ if $r_1\neq0$.

If we set $X=Z^{\ast}QVPU$, then $X$ is a unitary matrix in $\mathbb{DQ}^{n\times n}$ and $A_{5}=XAX^{\ast}$.

So far, we have completed the task of diagonalizing the original matrix $A$ under unitary equivalence. To make the diagonal entries of $A_{5}$ look more orderly, we now proceed with the following discussion. By Lemma \ref{2}, each diagonal element of matrix $A_{5}$ can be unitarily equivalent to a dual complex number. If we place in order along the diagonal all those unitary elements that convert the diagonal entries of $A_{5}$ into dual complex numbers, the resulting matrix is unitary, we denote it by $Y$. Then furthermore,
$$
\Sigma=YA_{5}Y^{\ast}=YXAX^{\ast}Y^{\ast}.
$$
Let $W=YX$. Then $\Sigma=WAW^{\ast}$. Therefore, under unitary equivalence, the matrix $A$ is finally transformed into the desired form.

By Lemma \ref{2}, when $r_i>0$, each diagonal entry
\[
r_i \boldsymbol{i}+\rho_{iq}\boldsymbol{i}\epsilon+\boldsymbol{j}\theta\epsilon
\]
is unitarily equivalent to
\[
r_i \boldsymbol{i}+\rho_{iq}\boldsymbol{i}\epsilon.
\]
When $r_1=0$, the first block is already
\[
\operatorname{diag}
(R_{11}\boldsymbol{i}\epsilon,\ldots,R_{1k_1}\boldsymbol{i}\epsilon).
\]
Hence $\Sigma$ has exactly the form in (\ref{9}), with
\[
R_{i1}\leq\cdots\leq R_{ik_i}.
\]

Let $\sigma_m$ be the $m$-th diagonal entry of $\Sigma$, and let
$e_m$ be the $m$-th standard basis vector. Since
\[
\Sigma e_m=e_m\sigma_m
\]
and $\Sigma=WAW^*$, we have
\[
A(W^*e_m)=W^*e_m\sigma_m.
\]
Moreover, $W^*e_m$ is appreciable because $W$ is unitary. Therefore,
every diagonal entry $\sigma_m$ of $\Sigma$ is a right eigenvalue of
$A$.

\end{proof}

\section{The Autonne--Takagi factorization of $\phi$-Hermitian dual quaternion matrices}
In this section, we propose the diagonalization of $\phi$-Hermitian dual quaternion matrices for $\phi\in\mathbb{S}^{3}$ under unitary $\phi$-congruence, which is regarded as a generalization of the Autonne--Takagi factorization of symmetric complex matrix and $\phi$-Hermitian quaternion matrix for $\phi\in\mathbb{S}^{3}$. We first introduce some definitions and their related properties that will be used later.

Assume $\phi\in\mathbb{Q}$, then $\phi$ can be expressed in the form $\phi_{0}+\phi_{1}\boldsymbol{i}+\phi_{2}\boldsymbol{j}+\phi_{3}\boldsymbol{k}$.
Recall that
$$
\mathbb{S}^{3}=\{ \phi\in\mathbb{Q}; |\phi|=1\}, \ \mathbb{S}=\{ \phi\in\mathbb{Q}; |\phi|=1, \phi_{0}=0\}.
$$
Then $\mathbb{S}^{3}\setminus \mathbb{S}=\{ \phi\in \mathbb{S}^{3}; \phi_{0}\neq 0\}$. For convenience, we denote
$\mathbb{S}^{\prime}=\mathbb{S}^{3}\setminus (\mathbb{S}\cup\{\pm 1\})$.

\begin{definition}
Assume $\phi\in\mathbb{S}^{3}$ and $A\in\mathbb{DQ}^{n\times n}$. Define the $\phi$-conjugation of $A$ as $\overline{A}^{\phi}=\overline{\phi} A\phi$, and the $\phi$-transpose of $A$ as $A^{*{\phi}}=\overline{A^*}^{\phi}$.
\end{definition}
Suppose that $A=A_{st}+A_{\mathcal{I}}\epsilon\in\mathbb{DQ}^{n\times n}$, then $\overline{A}^{\phi}=\overline{\phi} A_{st}\phi+\overline{\phi} A_{\mathcal{I}}\phi\epsilon=\overline{A_{st}}^{\phi}+\overline{A_{\mathcal{I}}}^{\phi}\epsilon$ and
$A^{*{\phi}}=A_{st}^{*{\phi}}+A_{\mathcal{I}}^{*{\phi}}\epsilon$. The $\phi$-conjugation has the following properties.

\begin{lemma}\label{14}
Assume $\phi\in \mathbb{S}^{3}$,  $A, B\in \mathbb{DQ}^{n\times n}$. Then the following statements hold.
\begin{enumerate}

    \item  $(\overline{A}^{\phi})^* = \overline{A^*}^{\phi}=A^{*{\phi}}$, $(A^{*{\phi}})^{*}=(A^*)^{*{\phi}}$.

    \item If $A$ is invertible, then $(\overline{A}^{\phi})^{-1}=\overline{A^{-1}}^{\phi}$, $(A^{*{\phi}})^{-1}=(A^{-1})^{*{\phi}}$.

    \item $\overline{A+B}^{\phi}=\overline{A}^{\phi}+\overline{B}^{\phi}$, $(A+B)^{*{\phi}}=A^{*{\phi}}+B^{*{\phi}}$.
    \item $\overline{AB}^{\phi}=\overline{A}^{\phi}\cdot\overline{B}^{\phi}$,  $(AB)^{*{\phi}}= B^{*{\phi}}A^{*{\phi}}$.

\item $\overline{A^T}^{\phi}=(\overline{A}^{\phi})^T=\overline{A^{*{\phi}}}=(\overline{A})^{*{\phi}}$.

\item $(A^{*{\phi}})^{T}=(A^{T})^{*{\phi}}=\overline{\overline{A}^{\phi}}=\overline{\overline{A}}^{\phi}  $.
\end{enumerate}

In particular, if $\phi\in \mathbb{S}$ and $A, B\in \mathbb{DQ}^{n\times n}$, then the following equations also hold,
$$
\overline{\overline{A}^{\phi}}^{\phi}=A=(A^{*{\phi}})^{*{\phi}}, \  \overline{A^{*{\phi}}}^{\phi}=(\overline{A}^{\phi})^{*{\phi}}=A^{*}.
$$
\end{lemma}

\begin{definition}
Assume $\phi\in\mathbb{S}^{3}$ and $A\in\mathbb{DQ}^{n\times n}$. If ${\overline{A}^{\phi}}=A^{\ast}$, then we say $A$ is $\phi$-Hermitian.
\end{definition}

Notice that for $A=A_{st}+A_{\mathcal{I}}\epsilon$, $A=A^{*{\phi}}$ is equivalent to the two equations $A_{st}=A_{st}^{*{\phi}}$ and $A_{\mathcal{I}}=A_{\mathcal{I}}^{*{\phi}}$ holding simultaneously. That is, $A$ is $\phi$-Hermitian if and only if both $A_{st}$ and $A_{\mathcal{I}}$ are $\phi$-Hermitian.
\begin{remark}\label{hh}
Generally speaking, when $\phi\in \mathbb{S}^{\prime}$, the equation $\overline{\overline{A}^{\phi}}^{\phi}=A$ does not hold for most $A\in \mathbb{DQ}^{n\times n}$. Fortunately,  the equation $A=A^{*{\phi}}(A=-A^{*{\phi}})$ is always equivalent to the equation $\overline{A}^{\phi}=A^{*}(\overline{A}^{\phi}=-A^{*})$. A similar situation also occurs in the preprint \cite{YHX}.
\end{remark}
\begin{definition}
Assume $\phi\in\mathbb{S}^{3}$ and $A,B\in\mathbb{DQ}^{n\times n}$. If there exists a unitary matrix $U\in\mathbb{DQ}^{n\times n}$ such that $UAU^{*{\phi}}=B$, then we say $A$ is unitarily $\phi$-congruent to $B$ and denote by $A\sim_{\phi-u} B$.
\end{definition}
It can be verified that unitary $\phi$-congruence is an equivalence relation. Moreover,
\begin{enumerate}
\item  if $\phi\in\mathbb{S}$, then $A\sim_{\phi-u} B \Leftrightarrow \phi A \sim_{u} \phi B \Leftrightarrow A\overline{\phi} \sim_{u}           B\overline{\phi}$,
\item  if $\phi\in\mathbb{S}^{\prime}$, then $A\sim_{\phi-u} B \Leftrightarrow A\overline{\phi} \sim_{u} B\overline{\phi}$.
\end{enumerate}

In what follows, we discuss the diagonalization of $\phi$-Hermitian dual quaternion matrices under unitary $\phi$-congruence by dividing the discussion into three cases according to the range of $\phi$. That is, $\phi=\pm 1$, $\phi\in\mathbb{S}$ and $\phi\in\mathbb{S}^{\prime}$. 

\begin{lemma}[\cite{QL}]\label{16}
Suppose that $A=A_{st}+A_{\mathcal{I}}\epsilon\in\mathbb{DQ}^{n\times n}$ is Hermitian. Then there exists a unitary matrix $V\in\mathbb{DQ}^{n\times n}$ and a diagonal matrix $\Sigma\in\mathbb{DQ}^{n\times n}$ such that
$$
\Sigma=V^{\ast}AV,
$$
where
\begin{equation}\label{17}
\Sigma=\mathrm{diag}(\lambda_{1}+\lambda_{11}\epsilon,\ldots,\lambda_{1}+\lambda_{1k_{1}}\epsilon,\lambda_{2}+\lambda_{21}\epsilon,\ldots,\lambda_{r}+\lambda_{rk_{r}}\epsilon)
\end{equation}
with the diagonal entries of $\Sigma$ being $n$ eigenvalues of $A$. $\lambda_{i}$ is a $k_{i}$-multiple right eigenvalue of $A_{st}$, and
$$
\lambda_{1}<\cdots<\lambda_{r},\ \  \lambda_{i1}\leq\cdots\leq\lambda_{ik_{i}},\ \ \lambda_{ij}\in\mathbb{R},\ \ \sum_{i=1}^{r}k_{i}=n.
$$ 
\end{lemma}

It can be verified that if $\phi=\pm 1$, then the $\phi$-conjugation $\overline{A}^{\phi}$ we have defined is equal to $A$, the $\phi$-Hermitian dual quaternion matrix is exactly Hermitian dual quaternion matrix and unitary $\phi$-congruence is exactly unitary equivalence. Therefore, for $\phi=\pm 1$, we have the following conclusion by Lemma \ref{16}. 
\begin{proposition}\label{4.1}
Assume $\phi=\pm 1$ and $A\in\mathbb{DQ}^{n\times n}$. If $A$ is $\phi$-Hermitian, then there exists a unitary matrix $V\in\mathbb{DQ}^{n\times n}$ and a diagonal matrix $\Sigma\in\mathbb{DQ}^{n\times n}$ such that
$$
\Sigma=VAV^{\ast\phi},
$$
where
\begin{equation}\label{177}
\Sigma=\mathrm{diag}(\lambda_{1}+\lambda_{11}\epsilon,\ldots,\lambda_{1}+\lambda_{1k_{1}}\epsilon,\lambda_{2}+\lambda_{21}\epsilon,\ldots,\lambda_{r}+\lambda_{rk_{r}}\epsilon),
\end{equation}
$$
\lambda_{1}<\cdots<\lambda_{r},\ \  \lambda_{i1}\leq\cdots\leq\lambda_{ik_{i}},\ \ \lambda_{ij}\in\mathbb{R},\ \ \sum_{i=1}^{r}k_{i}=n.
$$ 
\end{proposition}

The conclusion of the following lemma can be seen in the proof of Theorem $3$ in \cite{HZ}.
\begin{lemma}\label{3}
Assume $\phi\in\mathbb{S}$ and $A\in\mathbb{DQ}^{n\times n}$. Then $A$ is $\phi$-Hermitian if and only if $\phi A$ is skew-Hermitian.
\end{lemma}

Based on Theorem \ref{4} and Lemma \ref{3}, we obtain the diagonalization of $\phi$-Hermitian dual quaternion matrices under unitary $\phi$-congruence for $\phi\in\mathbb{S}$.
\begin{proposition}\label{5}
Assume $\phi\in\mathbb{S}$ and $B\in\mathbb{DQ}^{n\times n}$ is $\phi$-Hermitian. Then there exists a unitary matrix $V\in\mathbb{DQ}^{n\times n}$ and a diagonal matrix $\Gamma$ such that
\begin{equation}\label{13}
VBV^{*{\phi}}=\Gamma=\mathrm{diag}(r_{1}+R_{11}\epsilon,\ldots,r_{1}+R_{1m_{1}}\epsilon,r_{2}+R_{21}\epsilon,\ldots,r_{r}+R_{rm_{r}}\epsilon),
\end{equation}
where 
$$
0\leq r_{1}<\cdots<r_{r},\ \ R_{i1}\leq\cdots\leq R_{im_{i}},\ \ R_{ij}\in\mathbb{R},\ \ \sum_{i=1}^{r}m_{i}=n.
$$
\end{proposition}
\begin{proof}
 Suppose that $B$ is $\phi$-Hermitian. This is equivalent to the fact that $\phi B$ is skew-Hermitian by Lemma \ref{3}. Then we can find a unitary matrix $X\in\mathbb{DQ}^{n\times n}$ such that
$$
X(\phi B)X^{\ast}=\Sigma=\mathrm{diag}(r_{1}\boldsymbol{i}+R_{11}\boldsymbol{i}\epsilon,\ldots,r_{1}\boldsymbol{i}+R_{1m_{1}}\boldsymbol{i}\epsilon,r_{2}\boldsymbol{i}+R_{21}\boldsymbol{i}\epsilon,\ldots,r_{r}\boldsymbol{i}+R_{rm_{r}}\boldsymbol{i}\epsilon),
$$
where 
$$
0\leq r_{1}<\cdots<r_{r},\ \ R_{i1}\leq\cdots\leq R_{im_{i}},\ \ R_{ij}\in\mathbb{R},\ \ \sum_{i=1}^{r}m_{i}=n.
$$
By Lemma \ref{1}, $\boldsymbol{i}$ and $\phi$ are unitarily equivalent. Therefore, the diagonal matrix $\Sigma$ can be transformed into
$$
\Xi=\mathrm{diag}(r_{1}\phi+R_{11}\phi\epsilon,\ldots,r_{1}\phi+R_{1m_{1}}\phi\epsilon,r_{2}\phi+R_{21}\phi\epsilon,\ldots,r_{r}\phi+R_{rm_{r}}\phi\epsilon)
$$
through a unitary transformation. We may denote this process by $Y\Sigma Y^{\ast}=\Xi$, where $Y$ is unitary. Now we have
$$
YX(\phi B)X^{\ast}Y^{\ast}=\Xi,
$$
which is equivalent to
$$
\overline{(YX)}^{\phi}B(YX)^{\ast}=\overline{\phi}YX\phi BX^{\ast}Y^{\ast}=\overline{\phi}\Xi.
$$
Denote $\Gamma=\overline{\phi}\Xi$, then $\Gamma$ possesses the desired form in (\ref{13}). By Lemma \ref{14}, we may write $V=\overline{(YX)}^{\phi}$. Then $\overline{V}^{\phi}=YX$ and furthermore $V^{*{\phi}}=(YX)^{\ast}$. Thus, finally
$$
\ \ \ \ \ \ \ \ \ \ \ \ \ \ \ \ \ \ \ \ \ \ \ \ \ \ \ \ \ \ \ \ \ \ \ \ \ \ \ \ \ \ VBV^{*{\phi}}=\Gamma.\ \ \ \ \ \ \ \ \ \ \ \ \ \ \ \ \ \ \ \ \ \ \ \ \ \ \ \ \ \ \ \ \ \ \ \ \ \ \ \ \ \ \qedhere
$$
\end{proof}
As for the case $\phi\in\mathbb{S}^{\prime}$, we also obtain an analogous conclusion for $\phi$-Hermitian dual quaternion matrices. Here we use some facts that have already been proved in the preprint \cite{YHX}. 
\begin{definition}[\cite{YHX}]
Assume $\phi\in \mathbb{Q}\setminus\mathbb{R}$. Let $\phi=\phi_{0}+r_{\phi}I_{\phi}$, where 
$$
  r_{\phi}=|\rm{Im}\phi|>0,\ \  I_{\phi}=\frac{Im\phi}{|Im\phi|}  . 
$$  
  Define the set $\mathbb{C}_{I_{\phi}}$ to be
 $$
 \mathbb{C}_{I_{\phi}}=\{ a+bI_{\phi};a,b\in\mathbb{R}\}.
 $$
\end{definition}
Notice that, for a fixed $\phi\in \mathbb{Q}\setminus\mathbb{R}$, $\mathbb{C}_{I_{\phi}}$ is isomorphic to $\mathbb{C}$. We use the notation $\mathbb{C}_{I_{\phi}}^{n\times n}$ to denote the set of $n\times n$ quaternion matrices whose entries all belong to $\mathbb{C}_{I_{\phi}}$. That is,
$$
\mathbb{C}_{I_{\phi}}^{n\times n}=\{ A=(a_{ij})\in \mathbb{Q}^{n\times n};a_{ij}\in \mathbb{C}_{I_{\phi}} ~ for ~i,j=1,\ldots,n\}.
$$
From the form of $\mathbb{C}_{I_{\phi}}^{n\times n}$, it can be seen that $\mathbb{C}_{I_{\phi}}^{n\times n}\cong \mathbb{C}^{n\times n}$. Moreover, for a fixed matrix $B\in\mathbb{Q}^{n\times n}$, we have
\begin{enumerate}
\item $B\in \mathbb{C}_{I_{\phi}}^{n\times n}$ if and only if $B\in \mathbb{C}_{I_{\overline{\phi}}}^{n\times n}$,
\item $B\in \mathbb{C}_{I_{\phi}}^{n\times n}$ if and only if $B^{\ast}\in \mathbb{C}_{I_{\phi}}^{n\times n}$.
\end{enumerate}

In fact, for $\phi\in \mathbb{S}^{\prime}$, $\mathbb{C}_{I_{\phi}}^{n\times n}$ and $\mathbb{C}_{I_{\phi^{2}}}^{n\times n}$ are exactly the same set. Thus, for the sake of convenience, in what follows we shall predominantly use the symbol $\mathbb{C}_{I_{\phi}}^{n\times n}$ in place of $\mathbb{C}_{I_{\phi^{2}}}^{n\times n}$ when $\phi\in \mathbb{S}^{\prime}$.
\begin{lemma}[\cite{YHX}]\label{15}
Let $\phi\in \mathbb{S}^{\prime}$, $A\in \mathbb{Q}^{n\times n}$. If $A$ is $\phi$-Hermitian, then $A\in \mathbb{C}_{I_{\phi}}^{n\times n}$.
\end{lemma}

\begin{lemma}[\cite{YHX}]\label{wa}
Let $\phi\in \mathbb{S}^{\prime}$, $A\in \mathbb{Q}^{n\times n}$. If $A$ is $\phi$-Hermitian, then $A$ is a Hermitian matrix in $\mathbb{C}_{I_{\phi}}^{n\times n}$.
\end{lemma}

\begin{lemma}[\cite{YHX}]\label{66}
Let $\phi\in \mathbb{S}^{\prime}$, $A\in \mathbb{Q}^{n\times n}$. If $A$ is $\phi$-Hermitian, then there exists a unitary matrix $U\in \mathbb{C}_{I_{\phi}}^{n\times n}$ and a real diagonal matrix $\Delta$ such that
$$
\Delta=UA U^{\ast\phi}.
$$
\end{lemma}

The result in Lemma \ref{66} can be viewed as a generalization of the Autonne-Takagi factorization for $\eta$-Hermitian quaternion matrices where $\eta$ is pure imaginary.

Next, we show the diagonalization of $\phi$-Hermitian dual quaternion matrices under unitary $\phi$-congruence for $\phi\in \mathbb{S}^{\prime}$.
\begin{proposition}\label{55}
Let $\phi\in \mathbb{S}^{\prime}$, $A\in \mathbb{DQ}^{n\times n}$. If $A$ is $\phi$-Hermitian, then there exists a unitary matrix $W\in \mathbb{DQ}^{n\times n}$ and a diagonal matrix $\Lambda$ with dual numbers as its diagonal entries such that
\begin{equation}\label{144}
\Lambda=WA W^{\ast\phi}=\mathrm{diag}\bigl(
r_1+R_{11}\epsilon,\ldots,r_1+R_{1m_1}\epsilon,
r_2+R_{21}\epsilon,\ldots,
r_r+R_{rm_r}\epsilon
\bigr),
\end{equation}
where
\[
r_1<\cdots<r_r,\ \ R_{i1}\leq\cdots\leq R_{im_i},\ \ r_i,R_{ij}\in\mathbb{R},\ \ \sum_{i=1}^{r}m_i=n.
\]
\end{proposition}

\begin{proof}
Recall that for a matrix $A=A_{st}+A_{\mathcal{I}}\epsilon\in\mathbb{DQ}^{n\times n}$, $A$ is $\phi$-Hermitian if and only if both $A_{st}$ and $A_{\mathcal{I}}$ are $\phi$-Hermitian. Then both $A_{st}$ and $A_{\mathcal{I}}$ are in the set $\mathbb{C}_{I_{\phi}}^{n\times n}$ by Lemma \ref{15}. Therefore,
$$
A=A_{st}+A_{\mathcal{I}}\epsilon\in \mathbb{C}_{I_{\phi}}^{n\times n}+ \mathbb{C}_{I_{\phi}}^{n\times n}\epsilon,
$$
where
$$
\mathbb{C}_{I_{\phi}}^{n\times n}+ \mathbb{C}_{I_{\phi}}^{n\times n}\epsilon\cong \mathbb{C}^{n\times n}+ \mathbb{C}^{n\times n}\epsilon=\mathbb{DC}^{n\times n}\subseteq \mathbb{DQ}^{n\times n}.
 $$
 Hence $A$ is a Hermitian matrix in $\mathbb{C}_{I_{\phi}}^{n\times n}+ \mathbb{C}_{I_{\phi}}^{n\times n}\epsilon$ by Lemma \ref{wa}. Following from the construction of the unitary matrix in the proof of Lemma \ref{16} (\cite[Theorem 4.1]{QL}), if we replace every $\mathbb{Q}^{n\times n}$ in that process by $\mathbb{C}_{I_{\phi}}^{n\times n}$, then we obtain that there exists a unitary matrix $W\in \mathbb{C}_{I_{\phi}}^{n\times n}+ \mathbb{C}_{I_{\phi}}^{n\times n}\epsilon \subseteq \mathbb{DQ}^{n\times n}$
such that
$$
WAW^{\ast}=\Lambda,
 $$
 where $\Lambda$ is a diagonal matrix and all the diagonal entries of $\Lambda$ are dual numbers. Furthermore, by $W\in \mathbb{C}_{I_{\phi}}^{n\times n}+ \mathbb{C}_{I_{\phi}}^{n\times n}\epsilon$, we have
$$
\Lambda=WAW^{\ast}=WAW^{\ast}\overline{\phi}\phi=WA\overline{\phi}W^{\ast}\phi=WAW^{\ast\phi}.
$$
\end{proof}

Together with Propositions \ref{4.1}--\ref{55}, we can obtain the generalized Autonne--Takagi factorization of $\phi$-Hermitian dual quaternion matrices for $\phi\in\mathbb{S}^{3}$ as follows.
\begin{theorem}
Let $\phi\in\mathbb{S}^{3}$, $A\in \mathbb{DQ}^{n\times n}$. If $A$ is $\phi$-Hermitian, then there exists a unitary matrix $W\in \mathbb{DQ}^{n\times n}$ and a diagonal matrix $\Lambda$ with dual numbers as its diagonal entries such that
$$
\Lambda=WA W^{\ast\phi}.
$$
\end{theorem}
\begin{remark}
One should note that regardless of whether the real part of $\phi$ in $\mathbb{S}^{3}$ is zero, the $\phi$-Hermitian dual quaternion matrix $A$ can always be diagonalized to a dual real matrix $\Lambda$ in the sense of unitary $\phi$-congruence. However, when $\phi\in\mathbb{S}$ , the diagonal entries can be nonnegative, whereas when $\phi\in\mathbb{S}^{3}\setminus \mathbb{S}$, they can only be made real.
\end{remark}

\end{document}